\documentclass[10pt,oneside]{amsart}

\usepackage[
paper=a4paper,
headsep=20pt,text={32pc,50pc},centering,includehead
]{geometry}

\usepackage[bookmarks]{hyperref}
\usepackage{amssymb,amsxtra,dsfont}
\usepackage[all]{xy}
\usepackage{pb-diagram,pb-xy}
\usepackage{graphicx,color}
\usepackage{epstopdf}
\usepackage{pinlabel}
\usepackage{float}
\usepackage{array,calc,booktabs}
\usepackage[
  textwidth=1.2in,
  backgroundcolor=yellow,
  bordercolor=orange,
  textsize=small
]{todonotes}

\iftrue
\makeatletter
\def\@settitle{%
  \vspace*{0pt}
  \begin{flushleft}%
    \LARGE\bfseries
    \strut\@title\strut
  \end{flushleft}%
}
\def\@setauthors{%
  \begingroup
  \def\thanks{\protect\thanks@warning}%
  \trivlist
  \raggedright
  \large \@topsep27\p@\relax
  \advance\@topsep by -\baselineskip
  \item\relax
  \author@andify\authors
  \def\\{\protect\linebreak}%
  \authors
  \ifx\@empty\contribs
  \else
    ,\penalty-3 \space \@setcontribs
    \@closetoccontribs
  \fi
  \normalfont
  \endtrivlist
  \endgroup
}
\def\@setaddresses{\par
  \nobreak \begingroup
  \small\raggedright
  \def\author##1{\nobreak\addvspace\smallskipamount}%
  \def\\{\unskip, \ignorespaces}%
  \interlinepenalty\@M
  \def\address##1##2{\begingroup
    \par\addvspace\bigskipamount\noindent
    \@ifnotempty{##1}{(\ignorespaces##1\unskip) }%
    {\ignorespaces##2}\par\endgroup}%
  \def\curraddr##1##2{\begingroup
    \@ifnotempty{##2}{\nobreak\noindent\curraddrname
      \@ifnotempty{##1}{, \ignorespaces##1\unskip}\/:\space
      ##2\par}\endgroup}%
  \def\email##1##2{\begingroup
    \@ifnotempty{##2}{\nobreak\noindent E-mail address%
      \@ifnotempty{##1}{, \ignorespaces##1\unskip}\/:\space
      \ttfamily##2\par}\endgroup}%
  \def\urladdr##1##2{\begingroup
    \def~{\char`\~}%
    \@ifnotempty{##2}{\nobreak\noindent\urladdrname
      \@ifnotempty{##1}{, \ignorespaces##1\unskip}\/:\space
      \ttfamily##2\par}\endgroup}%
  \addresses
  \endgroup
  \global\let\addresses=\@empty
}
\def\@setabstracta{%
    \ifvoid\abstractbox
  \else
    \skip@17pt \advance\skip@-\lastskip
    \advance\skip@-\baselineskip \vskip\skip@
    \box\abstractbox
    \prevdepth\z@ 
    \vskip-15pt
  \fi
}
\renewenvironment{abstract}{%
  \ifx\maketitle\relax
    \ClassWarning{\@classname}{Abstract should precede
      \protect\maketitle\space in AMS document classes; reported}%
  \fi
  \global\setbox\abstractbox=\vtop \bgroup
    \normalfont\small
    \list{}{\labelwidth\z@
      \leftmargin0pc \rightmargin\leftmargin
      \listparindent\normalparindent \itemindent\z@
      \parsep\z@ \@plus\p@
      
    }%
    \item[\hskip\labelsep\bfseries\abstractname.]%
}{%
  \endlist\egroup
  \ifx\@setabstract\relax \@setabstracta \fi
}

\def\ps@headings{\ps@empty
  \def\@evenhead{%
    \setTrue{runhead}%
    \normalfont\scriptsize
    \rlap{\thepage}\hfill
    \def\thanks{\protect\thanks@warning}%
    \leftmark{}{}}%
  \def\@oddhead{%
    \setTrue{runhead}%
    \normalfont\scriptsize
    \def\thanks{\protect\thanks@warning}%
    \rightmark{}{}\hfill \llap{\thepage}}%
  \let\@mkboth\markboth
}\ps@headings

\def\section{\@startsection{section}{1}%
  \z@{-1.4\linespacing\@plus-.5\linespacing}{.8\linespacing}%
  {\normalfont\bfseries\Large}}
\def\subsection{\@startsection{subsection}{2}%
  \z@{-.8\linespacing\@plus-.3\linespacing}{.5\linespacing\@plus.2\linespacing}%
  {\normalfont\bfseries\large}}
\def\subsubsection{\@startsection{subsubsection}{3}%
  \z@{.7\linespacing\@plus.2\linespacing}{-1.5ex}%
  {\normalfont\bfseries}}
\def\@secnumfont{\bfseries}

\renewcommand\contentsnamefont{\bfseries}
\def\@starttoc#1#2{\begingroup
  \setTrue{#1}%
  \par\removelastskip\vskip\z@skip
  \@startsection{}\@M\z@{\linespacing\@plus\linespacing}%
    {.5\linespacing}{
      \contentsnamefont}{#2}%
  \ifx\contentsname#2%
  \else \addcontentsline{toc}{section}{#2}\fi
  \makeatletter
  \@input{\jobname.#1}%
  \if@filesw
    \@xp\newwrite\csname tf@#1\endcsname
    \immediate\@xp\openout\csname tf@#1\endcsname \jobname.#1\relax
  \fi
  \global\@nobreakfalse \endgroup
  \addvspace{32\p@\@plus14\p@}%
  \let\tableofcontents\relax
}
\def\contentsname{Contents}
\def\l@section{\@tocline{2}{.5ex}{0mm}{5pc}{}}
\def\l@subsection{\@tocline{2}{0pt}{2em}{5pc}{}}
\makeatother
\fi 

\def\to{\mathchoice{\longrightarrow}{\rightarrow}{\rightarrow}{\rightarrow}}
\makeatletter
\newcommand{\shortxra}[2][]{\ext@arrow 0359\rightarrowfill@{#1}{#2}}
\def\longrightarrowfill@{\arrowfill@\relbar\relbar\longrightarrow}
\newcommand{\longxra}[2][]{\ext@arrow 0359\longrightarrowfill@{#1}{#2}}

\def\addtagsub#1{\let\oldtf=\tagform@\def\tagform@##1{\oldtf{##1}\hbox{$_{#1}$}}}

\makeatother

\makeatletter
\def\Nopagebreak{\@nobreaktrue\nopagebreak}

\newtheoremstyle{theorem-giventitle}
        {}{}              
        {\itshape}                      
        {}                              
        {\bfseries}                     
        {.}                             
        {\thm@headsep}                             
        {\thmnote{\bfseries#3}}
\newtheoremstyle{theorem-givenlabel}
        {}{}              
        {\itshape}                      
        {}                              
        {\bfseries}                     
        {.}                             
        {\thm@headsep}                             
        {\thmname{#1}~\thmnumber{#3}\setcurrentlabel{#3}}

\newtheoremstyle{definition-giventitle}
        {}{}              
        {}                      
        {}                              
        {\bfseries}                     
        {.}                             
        {\thm@headsep}                             
        {\thmnote{\bfseries#3}}
\def\setcurrentlabel#1{\gdef\@currentlabel{#1}}

\makeatother

\newtheorem{theorem}{Theorem}[section]
\newtheorem{theoremalpha}{Theorem}
\newtheorem{corollaryalpha}[theoremalpha]{Corollary} 

\newtheorem{corollary}[theorem]{Corollary}

\theoremstyle{definition}

\theoremstyle{theorem-giventitle}
\newtheorem{theorem-named}{}
\theoremstyle{theorem-givenlabel}
\newtheorem{theorem-labeled}{Theorem}

\theoremstyle{definition-giventitle}
\newtheorem{definition-named}{}
\newtheorem{step-named}{}

\numberwithin{equation}{section}

\def\Z{\mathbb{Z}}

\def\R{\mathbb{R}}
\def\C{\mathbb{C}}

\def\N{\mathcal{N}}

\def\sign{\operatorname{sign}}

\def\lk{\operatorname{lk}}

\def\ldim{\dim^{(2)}}
\def\lsign{\sign^{(2)}}

\def\rhot{\rho^{(2)}}

\def\setminus{\smallsetminus}

\def\Vol{\operatorname{Vol}}

\begin{document}

\vspace*{0mm}

\title%
{Complexity of surgery manifolds and Cheeger-Gromov invariants}

\author{Jae Choon Cha}
\address{
  Department of Mathematics\\
  POSTECH\\
  Pohang 790--784\\
  Republic of Korea
  \quad -- and --\linebreak
  School of Mathematics\\
  Korea Institute for Advanced Study \\
  Seoul 130--722\\
  Republic of Korea
}
\email{jccha@postech.ac.kr}

\def\subjclassname{\textup{2010} Mathematics Subject Classification}
\expandafter\let\csname subjclassname@1991\endcsname=\subjclassname
\expandafter\let\csname subjclassname@2000\endcsname=\subjclassname
\subjclass{%
}


\begin{abstract}
  We present new lower bounds on the complexity of Dehn surgery
  manifolds of knots, using our recent result on the Cheeger-Gromov
  rho invariants and triangulations.  As an application, we give
  explicit examples of closed hyperbolic 3-manifolds with fixed first
  homology for which the gap between the Gromov norm and the
  complexity is arbitrarily large.
\end{abstract}

\maketitle

\section{Main results}

In a recent paper~\cite{Cha:2014-1}, we have presented explicit linear
universal bounds of the Cheeger-Gromov $L^2$ $\rho$-invariants of
3-manifolds in terms of topological descriptions, especially
triangulations of 3-manifolds.  In this paper, we use the results
of~\cite{Cha:2014-1} to study the complexity of surgery manifolds.

\subsection*{Complexity of surgery manifolds of knots}

For a 3-manifold $M$, the \emph{complexity} $c(M)$ is defined by
\[
c(M) := \min \{\text{the number of 3-simplices in a pseudo-simplicial
  triangulation of $M$}\}.
\]
Here, a pseudo-simplicial triangulation designates a collection of
disjoint 3-simplices together with affine identifications of their
faces in pairs, whose quotient space is homeomorphic to the manifold.
(Sometimes it is just called a triangulation.)  It is more flexible
than a simplicial complex structure and studied extensively in the
3-manifold literature.  This definition of $c(M)$ is equivalent to
Matveev's complexity~\cite{Matveev:1990-1} for closed irreducible
3-manifolds $M\ne S^3$, $\R P^3$,~$L(3,1)$.

The determination of $c(M)$ is regarded as a hard
problem~\cite{Jaco-Rubinstein-Tillman:2013-1}.  In particular the main
difficulty is in finding an efficient lower bound.

As the first result of this paper, we give new lower bounds for the
complexity for Dehn surgery manifolds of knots, which are strong
enough to determine the complexity asymptotically.

To state it, we use the following notations.  For a knot $K$ in $S^3$
and $n\in \Z$, let $M(K,n)$ be the 3-manifold obtained by Dehn surgery
on $K$ with slope~$n$.  We denote by $c(K)$ the \emph{crossing number}
of $K$, that is, the minimal number of crossings in a planar diagram
of~$K$.
For two functions $f(n)$ and $g(n)$, we write $f(n)\in \Theta(g(n))$
if
\[
0 < \liminf_{n\to \infty} \frac{|f(n)|}{|g(n)|} \quad\text{and}\quad
\limsup_{n\to \infty} \frac{|f(n)|}{|g(n)|} < \infty,
\]
that is, the asymptotic growth of $f(n)$ and $g(n)$ are equivalent.
In terms of the big $O$ notation, $f(n)\in \Theta(g(n))$ if and only
if $f(n)\in O(g(n))$ and $g(n)\in O(f(n))$.

\begin{theoremalpha}
  \label{theorem:complexity-lower-bound-for-integral-dehn-surgery}
  For any knot $K$ in $S^3$ and any integer $n\ne 0$,
  \[
  \frac{|n|-3-6c(K)}{627419520} \le c(M(K,n)) \le 96|n|+128c(K).
  \]  
  Consequently, $c(M(K,n)) \in \Theta(n)$, and
  $\lim\limits_{|n|\to\infty} c(M(K,n)) = \infty$.
\end{theoremalpha}


The upper bound for $c(M(K,n))$ in
Theorem~\ref{theorem:complexity-lower-bound-for-integral-dehn-surgery}
is an immediate consequence of~\cite[Theorem~A]{Cha:2015-1}.  We give
a proof of the lower bound in
Section~\ref{section:proof-of-lower-bound}, using results on the
Cheeger-Gromov invariants and triangulations in~\cite{Cha:2014-1}.

We remark that while the statement of
Theorem~\ref{theorem:complexity-lower-bound-for-integral-dehn-surgery}
(and Theorem~\ref{theorem:complexity-lower-bound-from-signature}
stated below) is purely 3-dimensional, its proof, including the theory
developed in \cite{Cha:2014-1}, is essentially \emph{4-dimensional}.

We remark that in~\cite{Cha:2014-1} a lower bound
\begin{equation}
  \label{equation:lens-space-lower-bound}
  \frac{|n|-3}{627419520} \le c(L(n,1)).
\end{equation}
for the lens spaces $L(n,1)$ was obtained, and was shown to be
arbitrarily larger than lower bounds from previously known methods for
any large odd~$n$.  Our
Theorem~\ref{theorem:complexity-lower-bound-for-integral-dehn-surgery}
subsumes~\eqref{equation:lens-space-lower-bound} as a special case.

Note that the lower bound in
Theorem~\ref{theorem:complexity-lower-bound-for-integral-dehn-surgery}
is useful when the surgery slope $n$ is large.  Our method also
presents another lower bound, which is given in terms of classical
\emph{signature invariants} of the given knot~$K$.  To state this
result, we need the following notation.  Recall that for a knot $K$ in
$S^3$, a Seifert matrix $A$ is defined by choosing a Seifert surface
$F$ and a basis of~$H_1(F)$.  The \emph{Levine-Tristram signature
  function} of $K$ is defined by
\[
\sigma_K(\omega) :=
\sign\big((1-\omega)A+(1-\overline{\omega})A^T\big), \quad \omega\in
S^1\subset \C.
\]
We denote the average of the signature function over the $d$th roots
of unity by
\[
\bar\sigma(L,d) := \frac{1}{d} \sum_{k=1}^{d-1}
\sigma_{L}(e^{2\pi k\sqrt{-1}/d}).
\]

\begin{theoremalpha}
  \label{theorem:complexity-lower-bound-from-signature}
  For any knot $K$ in~$S^3$ and for any integer $n\ne 0$,
  \[
  c(M(K,n)) \ge \frac{3\big|\bar\sigma(K,|n|)\big|-|n|+1 \mathstrut}{627419520}.
  \]
\end{theoremalpha}



In what follows we discuss an application.

\subsection*{Gromov norm, stable complexity, and complexity}

Recall that the \emph{Gromov norm} (or \emph{simplicial volume}) of a
closed orientable manifold $M$ is defined by
\[
\|M\| := \inf \{ \|c\| : \text{$c$ is a real singular cycle
  representing the fundamental class of $M$} \}
\]
where $\|c\| := \sum |a_\sigma|$ denotes the $L^1$-norm of a chain
$c=\sum_\sigma a_\sigma\sigma$~\cite{Gromov:1982-1}.  Due to Gromov
and Thurston, if $M$ is a closed hyperbolic 3-manifold, then $\|M\| =
\Vol(M)/v_3$ where $v_3=1.01494\cdots$ is the volume of a regular
ideal tetrahedron in~$\mathbb{H}^3$~\cite{Thurston:1978-1}.

Since we allow more flexible simplices for the Gromov norm when it is
compared with the complexity, we obtain immediately the following (see
also~\cite[Proposition~2.1]{Matveev-Petronio-Vesnin:2009-1}):
\begin{equation}
  \label{equation:gromov-norm-complexity-inequality}
  \|M\| \le c(M).
\end{equation}


Using Thurston's hyperbolic Dehn surgery theorem, it can be observed
that the gap $c(M)-\|M\|$ of
\eqref{equation:gromov-norm-complexity-inequality} can be large: for
any hyperbolic knot, Dehn surgery along large slopes gives infinitely
many hyperbolic 3-manifolds with bounded Gromov norm.  Since there are
only finitely many 3-manifolds with fixed complexity, it follows that
these Dehn surgery manifolds have unbounded complexity.  Note that
this does not give us explicit examples with large gap nor any
estimate of the gap $c(M)-\|M\|$.

We remark that in spite of this the inequality
\eqref{equation:gromov-norm-complexity-inequality} has been used as a
practically useful lower bound for $c(M)$ for hyperbolic 3-manifolds
in the literature, implicitly expecting that
\eqref{equation:gromov-norm-complexity-inequality} is not far from
optimal in the hyperbolic case.  For instance
see~\cite{Matveev-Petronio-Vesnin:2009-1, Petronio-Vesnib:2009-1}.




Using our method, we give explicit examples of closed hyperbolic
3-manifolds $M$ with fixed first homology, for which the gap
$c(M)-\|M\|$ is estimated to be large.  Let $J_n$ be the 2-bridge knot
in Figure~\ref{figure:two-bridge-example}.

\begin{figure}[H]
  \labellist\small
  \pinlabel {$2n$ crossings} at 148 14
  \endlabellist
  \includegraphics[scale=.85]{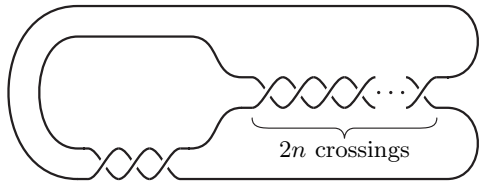}
  \caption{A 2-bridge knot $J_n$}
  \label{figure:two-bridge-example}
\end{figure}

\begin{theoremalpha}
  \label{theorem:gromov-norm-complexity-gap}
  Suppose $d>1$ is a fixed integer, and let $M_n = M(J_n,d)$.
  Then for any $n>2$, $M_n$ is hyperbolic and
  \[
  c(M_n) - \|M_n\| > \frac{1}{627419520}\Big(3\Big(1-\frac{1}{d^2}\Big)n - (d+7)\Big) - 6.
  \]
  Consequently $\lim\limits_{n\to\infty} c(M_n) - \|M_n\|= \infty$.
\end{theoremalpha}


The proof of Theorem~\ref{theorem:gromov-norm-complexity-gap} relies
on Theorem~\ref{theorem:complexity-lower-bound-from-signature}\@, the
$2\pi$-theorem of Thurston and Gromov, and computer-assisted
computation using SnapPy~\cite{SnapPy}.  For details, see
Section~\ref{section:computation-for-examples}.

Note that $H_1(M_n)=\Z_d$ is fixed for the 3-manifolds $M_n$ in
Theorem~\ref{theorem:gromov-norm-complexity-gap}\@.  We remark that
for odd $d$, it follows that previously known lower bounds do not
detect the divergence of~$c(M_n)$.  Specifically, the lower bounds for
$c(M_n)$ obtained from the $\Z_2$-Thurston norm (and the double cover)
in~\cite{Jaco-Rubinstein-Tillman:2009-1,
  Jaco-Rubinstein-Tillman:2011-1, Jaco-Rubinstein-Tillman:2013-1}
vanish.  Also, the lower bound for $c(M_n)$ obtained from the first
homology in~\cite{Matveev-Pervova:2001-1} is a constant ($=\log_5 d$);
especially it vanishes if $d<5$.

In their study of characteristic numbers of
3-manifolds~\cite{Milnor-Thurston:1977-1}, Milnor and Thurston
introduced the notion of \emph{stable complexity}, which is defined by
\[
\sigma(M)=\inf\Big\{\frac{c(\tilde M)}{k} \,\Big|\, \tilde M
\rightarrow M \text{ is a finite cover of degree }k \Big\}.
\]
It is known that $\|M\| \le \sigma(M) \le c(M)$ (e.g.,
see~\cite{Francaviglia-Frigerio-Martelli:2012-1}).  By Francaviglia,
Frigerio, and Martelli
\cite[Proposition~1.6]{Francaviglia-Frigerio-Martelli:2012-1}, there
is a constant $D_3$ such that $\sigma(M) \le D_3\cdot \|M\|$ for any
closed hyperbolic 3-manifold~$M$.  From this and
Theorem~\ref{theorem:gromov-norm-complexity-gap}, we immediately
obtain that the gap between the complexity and the stable complexity
can be arbitrarily large for closed hyperbolic 3-manifolds.  More
specifically, we have:

\begin{corollaryalpha}
  If the 3-manifolds $M_n$ are as in
  Theorem~\ref{theorem:gromov-norm-complexity-gap}\@,
  $\lim\limits_{n\to\infty} c(M_n) - \sigma(M_n)=\infty$.
\end{corollaryalpha}

\subsection*{Acknowledgements}

This work was partially supported by NRF grants 2013067043 and
2013053914.

\section{Proofs of complexity estimates}
\label{section:proof-of-lower-bound}

We begin with definitions and some preparations in
Sections~\ref{subsection:cheeger-gromov-invariant-complexity}
and~\ref{subsection:finite-cyclic-cheeger-gromov-invariant}.  In
Section~\ref{subsection:lower-bound-surgery-manifold-complexity}, we
will prove the lower bounds in
Theorems~\ref{theorem:complexity-lower-bound-for-integral-dehn-surgery}
and~\ref{theorem:complexity-lower-bound-from-signature}.

\subsection{Cheeger-Gromov rho invariants and complexity}
\label{subsection:cheeger-gromov-invariant-complexity}

For a closed $(4k-1)$-dimensional smooth manifold $M$ and a
homomorphism $\phi\colon \pi_1(M) \to G$, Cheeger and Gromov defined
the $L^2$ $\rho$-invariant $\rhot(M,\phi)\in \R$ as the difference of
the ordinary and von Neumann $L^2$ eta invariants of the odd signature
operator~\cite{Cheeger-Gromov:1985-1}.  It is known that
$\rhot(M,\phi)$ can be defined as the $L^2$-signature defect of an
appropriate bounding $4k$-manifold.  (This also enables us to define
$\rhot(M,\phi)$ for topological manifolds.)

For the use in
Section~\ref{subsection:finite-cyclic-cheeger-gromov-invariant}, we
describe the $L^2$-signature defect definition of $\rhot(M,\phi)$ for
3-manifolds.  Given $\phi\colon \pi_1(M) \to G$, it is known that
there is an embedding of $G$ into another group $\Gamma$ and a
4-manifold $W$ bounded by $M$ for which there is a commutative diagram
as follows:
\begin{equation}
  \label{equation:bounding-manifold-diagram}
  \vcenter{\xymatrix@M=2mm{
    \pi_1(M) \ar[r]^{\phi} \ar[d]_{i_*} & G \ar@{^{(}->}[d]
    \\
    \pi_1(W) \ar@{..>}[r] & \Gamma
  }}
\end{equation}
In fact, by Kan and Thurston~\cite{Kan-Thurston:1976-1}, any $G$
embeds into an acyclic group $\Gamma$, and by bordism theory, such $W$
exists whenever $\Gamma$ is acyclic.

We consider the homology $H_*(W;\N\Gamma)$ of $W$ with coefficients in
the \emph{group von Neumann algebra} $\N\Gamma$ of $\Gamma$, which is
an algebra containing the group ring~$\C\Gamma$.  By Poincar\'e
duality, the intersection form
\[
\lambda\colon H_2(W;\N\Gamma)\times H_2(W;\N\Gamma) \to \N\Gamma
\]
is defined.  We invoke a couple of useful properties of~$\N\Gamma$;
first there is a spectral decomposition $H_2(W;\N\Gamma)=V_+\oplus
V_-\oplus V_0$, where $\lambda$ is positive definite, negative
definite, and zero on $V_+$, $V_-$, and $V_0$, respectively.
Second, the \emph{$L^2$-dimension}
$\ldim_\Gamma V \in \R\cup\{\infty\}$ is defined for any module $V$
over~$\N\Gamma$.  In our case, it turns out that $\ldim_\Gamma
V_{\pm}$ is finite.  We define the \emph{$L^2$-signature} of $W$ over
$\Gamma$ by
\begin{equation}
  \label{equation:l2-sign-definition}
  \lsign_\Gamma(W) := \ldim_\Gamma V_+ - \ldim_\Gamma V_-.
\end{equation}
Now the Cheeger-Gromov invariant of $(M,\phi)$ is defined by
\[
\rhot(M,\phi) := \lsign_\Gamma(W)-\sign(W)
\]
where $\sign(W)$ designates the ordinary signature of~$W$.  It is
known that $\rhot(M,\phi)$ is independent of the choice of the
diagram~\eqref{equation:bounding-manifold-diagram}.  For more details,
the readers are referred to, for instance,
\cite[Section~2.1]{Cha:2014-1}, \cite{Chang-Weinberger:2003-1}.

\subsubsection*{Lower bounds on complexity}

In~\cite{Cheeger-Gromov:1985-1}, Cheeger and Gromov presented a deep
analytic argument which shows that for each $M$, there is a universal
bound for the values of $\rhot(M,\phi)$.  That is, there is a constant
$C_M$ such that $|\rhot(M,\phi)|\le C_M$ (for any $G$ and) for any
$\phi\colon \pi_1(M) \to G$.  In~\cite{Cha:2014-1}, we developed a
topological approach to the universal bound.  One of the main result
was the following explicit linear bound in terms of the complexity:

\begin{theorem}[{\cite[Corollary~1.11]{Cha:2014-1}}]
  \label{theorem:lower-bound-of-complexity}
  If $M$ is a closed $3$-manifold, then for any homomorphism $\phi$ of
  $\pi_1(M)$,
  \[
  |\rhot(M,\phi)| \le 209139840\cdot c(M).
  \]
\end{theorem}

In the proof of
Theorems~\ref{theorem:complexity-lower-bound-for-integral-dehn-surgery}
and~\ref{theorem:complexity-lower-bound-from-signature} we will use
Theorem~\ref{theorem:lower-bound-of-complexity} as a lower bound
for~$c(M)$.  To estimate the lower bound, we will compute certain
$\rho$-invariants over finite cyclic groups, using a method for
general 3-manifolds discussed in the next subsection.

\subsection{Cheeger-Gromov invariants over finite cyclic groups}
\label{subsection:finite-cyclic-cheeger-gromov-invariant}

In this subsection we give an explicit formula for the Cheeger-Gromov
invariant of a 3-manifold over a finite cyclic group
(Theorem~\ref{theorem:rho-over-finite-cyclic-group}).  In fact, in
this case, $\rhot(M,\phi)$ can be obtained from the Atiyah-Singer
invariants~\cite{Atiyah-Singer:1968-3}, which are essentially
equivalent to the Casson-Gordon
invariants~\cite{Casson-Gordon:1978-1}.  Our proof depends on results
of Casson-Gordon~\cite{Casson-Gordon:1978-1} and
Gilmer~\cite{Gilmer:1981-1}, which in turn can be shown using the
Atiyah-Singer $G$-signature theorem.


Suppose $K$ is an oriented knot in $S^3$ and $n$ and $r$ are integers.
We call a union of finitely many disjoint parallels of $K$ an
\emph{$n$-twisted $r$-cable} if each parallel is taken along the
$n$-framing of $K$ and oriented in such a way that the sum of the
parallels is homologous to $rK$ in a tubular neighborhood of~$K$.

\begin{theorem}
  \label{theorem:rho-over-finite-cyclic-group}
  Suppose $M$ is a closed $3$-manifold, and $L = K_1\sqcup\cdots\sqcup
  K_r$ is an $r$-component oriented link in $S^3$ such that surgery on
  $L$ with integral coefficients $n_1,\ldots, n_r$ gives~$M$.  Let
  $\Lambda=(n_{ij})$ be the linking matrix defined by $n_{ii}=n_i$ and
  $n_{ij}=\lk(K_i,K_j)$ for $i\ne j$.  Suppose $\phi\colon \pi_1(M)\to
  \Z_d$ is a homomorphism.  Let $\mu_i$ be the positive meridian of
  $K_i$, and let $r_i$ be an integer satisfying $r_i = \phi(\mu_i)$
  in~$\Z_d$.  Let $L'$ be the link obtained from $L$ by replacing each
  component $K_i$ with a nonempty $n_i$-twisted $r_i$-cable of~$K_i$.
  Then we have
  \[
  \rhot(M,\phi) = \bar\sigma(L',d) - \frac{d-1}{d} \sign \Lambda +
  \frac{d^2-1}{3d^2} \sum_{i,j} r_i r_j n_{ij}.
  \]
\end{theorem}

\begin{proof}
  In~\cite{Casson-Gordon:1986-1}, Casson and Gordon defined an
  invariant $\sigma_r(M,\phi)$ as follows.  Since $\Omega_3(\Z_d)$ is
  finite, there is a $4$-manifold $W$ over $\Z_d$ such that $\partial W
  = sM$ over $\Z_d$ for some integer $s\ne 0$.  Let $\tilde W$ be the
  $\Z_d$-cover of~$W$.  The generator $1\in \Z_d$ induces an order $d$
  linear operator $g\colon H_2(\tilde W;\C) \to H_2(\tilde W;\C)$.
  Let $\sigma_k(\tilde W)$ be the signature of the intersection form
  of $\tilde W$ restricted on the $e^{2\pi k\sqrt{-1}/d}$-eigenspace
  of~$g$.  Then the rational number
  \begin{equation}
    \label{equation:definition-of-casson-gordon-invariant}
    \sigma_k(M,\phi):=\frac{1}{s}\big(\sigma_k(\tilde W)-\sign(W)\big)
  \end{equation}
  is well-defined, independent of the choice of~$W$.  (Our sign
  convention is opposite of that of~\cite{Casson-Gordon:1986-1} but
  agrees with that of~\cite{Gilmer:1981-1}.)  It is known that
  $\sigma_0(M,\phi)=0$; e.g., see \cite[p.~40]{Casson-Gordon:1986-1}.
  Due to Gilmer~\cite[Theorem~3.6]{Gilmer:1981-1}, we have
  \begin{equation}
    \label{equation:gilmer-formula}
    \sigma_k(M,\phi) = \sigma_{L'}(e^{2\pi k\sqrt{-1}/d}) - \sign \Lambda +
    \frac{2(d-k)k}{d^2} \sum_{i,j}r_i r_j n_{ij}
  \end{equation}
  for $0<k<d$.  See also \cite[Section~3]{Casson-Gordon:1978-1} for a
  special case.

  It is related to $\rhot(M,\phi)$ as follows.  Since $\Z_d$ is
  finite, the group von Neumann algebra $\N\Z_d$ is equal to the
  ordinary group ring $\C[\Z_d]$ and the $L^2$-dimension over $\Z_d$
  is given by $\ldim_{\Z_d} V = \frac{1}{d} \dim_{\C} V$.  Observe
  that $H_2(W;\N\Z_d) = H_2(W;\C[\Z_d]) \cong H_2(\tilde W;\C)$ is the
  orthogonal sum of the $e^{2\pi k\sqrt{-1}/d}$-eigenspaces
  ($k=0,\ldots,d-1$).  From the definitions of $\sigma_k(\tilde W)$
  and $\lsign_{\Z_d}(W)$ (see \eqref{equation:l2-sign-definition} in
  Section~\ref{subsection:cheeger-gromov-invariant-complexity}), it
  follows that
  \[
  \lsign_{\Z_d} W = \frac{1}{d} \sum_{k=0}^{d-1}\sigma_k(\tilde W).
  \]
  From this and
  \eqref{equation:definition-of-casson-gordon-invariant}, it follows
  that
  \begin{equation}
    \label{equation:rho-as-sum-of-casson-gordon}
    \rhot(M,\phi) = \frac1s \big(\lsign_{\Z_d}W - \sign W\big) = \frac{1}{d}
    \sum_{k=0}^{d-1} \sigma_k(M,\phi).
  \end{equation}
  Substituting \eqref{equation:gilmer-formula} into
  \eqref{equation:rho-as-sum-of-casson-gordon}, we obtain the desired
  formula for $\rhot(M,\phi)$.
\end{proof}

Using Theorem~\ref{theorem:rho-over-finite-cyclic-group}, we give an
explicit formula for the Cheeger-Gromov invariant of surgery manifolds
of knots.

\begin{corollary}
  \label{corollary:rho-for-knot-surgery}
  Suppose $K$ is a knot in $S^3$, $n>0$, and let
  \[
  \phi\colon \pi_1(M(K,n))\to H_1(M(K,n))=\Z_n
  \]
  be the abelianization.  Then
  \[
  \rhot(M(K,n),\phi) = \frac{n}{3} + \frac{2}{3n} - 1 +
  \bar\sigma(K,n).
  \]
\end{corollary}

\begin{proof}
  Note that $\phi$ takes a meridian of $K$ to $1\in \Z_n$.  Also, the
  $n$-twisted $1$-cable of $K$ is $K$ itself, and the linking matrix
  for the $n$-surgery on $K$ is
  $\Lambda=\begin{bmatrix}n\end{bmatrix}$.  Therefore, by applying
  Theorem~\ref{theorem:rho-over-finite-cyclic-group}, we obtain the
  formula for $\rhot(M(K,n),\phi)$.
\end{proof}

\subsection{Lower bounds of the complexity of surgery manifolds of knots}
\label{subsection:lower-bound-surgery-manifold-complexity}

For a knot $K$ in $S^3$, we denote by $g_4(K)$ the (topological)
\emph{slice genus} of~$K$.  That is, $g_4(K)$ is the minimal genus of
a properly embedded locally flat orientable surface in $B^4$ bounded
by~$K$.

\begin{theorem}
  \label{theorem:complexity-lower-bound-for-knot-surgery}
  Suppose $K$ is a knot in~$S^3$, and $n\ne 0$.  Then
  \[
  c(M(K,n)) \ge \frac{|n|-3-6g_4(K)}{627419520 \mathstrut}.
  \]
\end{theorem}

The slice genus $g_4(K)$ in
Theorem~\ref{theorem:complexity-lower-bound-for-knot-surgery} can be
replaced by either one of the smooth slice genus, the unknotting
number, the Seifert genus, or the crossing number of $K$, since
$g_4(K)$ is not greater than any one of them.  In particular, the
lower bound part of
Theorem~\ref{theorem:complexity-lower-bound-for-integral-dehn-surgery}
in the introduction is an immediate consequence of
Theorem~\ref{theorem:complexity-lower-bound-for-knot-surgery}.

\begin{proof}[Proof of Theorem~\ref{theorem:complexity-lower-bound-for-knot-surgery}]
  We may assume that $n>0$, by taking the mirror image of $K$ if
  $n<0$.  Let $\phi\colon \pi_1(M(K,n))\to \Z_n$ be the
  abelianization.  From
  Corollary~\ref{corollary:rho-for-knot-surgery}, we obtain
  \begin{equation}
    \label{equation:rhot-surgery-first-estimate}
    |\rhot(M(K,n),\phi)| \ge \frac{n-3-3|\bar\sigma(K,n)|}{3}.
  \end{equation}

  It is known that $|\sigma_K(\omega)| \le 2g_4(K)$ for any root of
  unity $\omega\in S^1$ (for example, see
  \cite[p.~145]{Taylor:1979-1}).  From this it follows that
  \begin{equation}
    \label{equation:signature-average-4-ball-genus}
    |\bar\sigma(K,n)| \le \frac{1}{n} \sum_{k=0}^{n-1} \big|\sigma_K(e^{2\pi k
      \sqrt{-1}})\big| \le 2g_4(K).
  \end{equation}

  From \eqref{equation:rhot-surgery-first-estimate} and
  \eqref{equation:signature-average-4-ball-genus}, it follows that
  \begin{equation}
    \label{equation:degree-lower-bound-for-rhot-surgery}
    |\rhot(M(K,n),\phi)| \ge \frac{ n-3-6g_4(K) }{3}.
  \end{equation}

  By combining Theorem~\ref{theorem:lower-bound-of-complexity} and
  \eqref{equation:degree-lower-bound-for-rhot-surgery}, we obtain the
  inequality.
\end{proof}

\begin{proof}[Proof of Theorem~\ref{theorem:complexity-lower-bound-from-signature}]
  Similarly to the above proof of
  Theorem~\ref{theorem:complexity-lower-bound-for-knot-surgery}, we
  may assume $n>0$.  Then it is easily verified that $\frac{n}{3}
  +\frac{2}{3n} -1$ is nonnegative.  From this and from
  Corollary~\ref{corollary:rho-for-knot-surgery}, it follows that
  \begin{equation}
    \label{equation:signature-lower-bound-for-rhot-surgery}
    \begin{aligned}
      |\rhot(M(K,n),\phi)| & \ge |\bar\sigma(K,n)| - \Big(\frac{n}{3}
      +\frac{2}{3n} -1 \Big)
      \\
      & \ge \frac{3|\bar\sigma(K,n)| - n +1}{3}.
    \end{aligned}
  \end{equation}

  By combining Theorem~\ref{theorem:lower-bound-of-complexity} and
  \eqref{equation:signature-lower-bound-for-rhot-surgery}, we obtain
  \[
  c(M(K,n)) \ge \frac{3|\bar\sigma(K,n)|-n+1 \mathstrut}{627419520}.
  \qedhere
  \]
\end{proof}

\section{Computation for examples}
\label{section:computation-for-examples}

This section is devoted to the proof of
Theorem~\ref{theorem:gromov-norm-complexity-gap}\@.  In the first
subsection we show that the surgery manifolds considered in
Theorem~\ref{theorem:gromov-norm-complexity-gap} are hyperbolic using
Gromov-Thurston's $2\pi$-theorem, and estimate the Gromov norm.  In
the second subsection we estimate the complexity of $M(J_n,d)$ via
computation of signature invariants, using
Theorem~\ref{theorem:complexity-lower-bound-from-signature}\@.

\subsection{Hyperbolicity and Gromov norm}
\label{subsection:hyperbolicity-gromov-norm}

Consider the link $L$ shown in
Figure~\ref{figure:two-bridge-seed-link}.

\begin{figure}[H]
  \includegraphics[scale=.9]{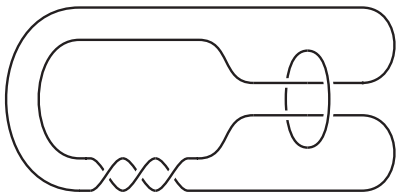}
  \caption{A two-component hyperbolic link $L$ with volume $5.3335$}
  \label{figure:two-bridge-seed-link}
\end{figure}

According to computation using SnapPy~\cite{SnapPy}, $L$ is a
hyperbolic link with $\Vol(S^3\setminus L)=5.3335$.  Also, for both
cusps, the translation lengths of the meridian and the longitude (on
the boundary of a maximal horoball neighborhood) are
$-0.4204+1.1124\sqrt{-1}$ and $3.3636$, respectively.

It follows that the length of the slope $k\in \Z$ is greater than
$6.7271$ if $k> 5$, and the length of the slope $1/n$ with $n\in \Z$
is greater than $6.4040$ if $n > 1$.  By the $2\pi$-theorem of Gromov
and Thurston and by the geometrization conjecture, it follows that the
$(k,1/n)$-surgery on the link $L$ is hyperbolic if $k > 5$ and $n >
1$, and by the hyperbolic Dehn surgery theorem, its Gromov norm is
bounded by $\Vol(S^3\setminus L)/v_3=5.2552$.

Observe that the $(1/n)$-surgery on the component drawn as a circle in
Figure~\ref{figure:two-bridge-seed-link} introduces $n$ left-handed
full twists between the two enclosed strands, so that the other
component becomes the 2-bridge knot $J_n$ shown in
Figure~\ref{figure:two-bridge-example}.  Also, the $(1/n)$-surgery
alters the framing of the other component by $-4n$ (for this we use
that the two enclosed strands have the same orientation).  It follows
that the surgery manifold $M(J_n,d)$ is equal to the
$(d+4n,1/n)$-surgery on the link $L$.

Summarizing, we have shown the following: $M(J_n,d)$ is hyperbolic
with $\|M(J_n,d)\|<5.2552$ for any $d>1$ and $n>1$.

\subsection{Complexity}

To obtain a lower bound for the complexity of $M(J_n,d)$ using
Theorem~\ref{theorem:complexity-lower-bound-from-signature}, we need
to compute the Levine-Tristram signature function of the 2-bridge
knot~$J_n$.  Using the Seifert surface and the generators $x_i$ shown
in Figure~\ref{figure:two-bridge-seifert-surface}, we obtain a
$2n\times 2n$ Seifert matrix $A$ for~$J_n$:
\[
A = 
\begin{bmatrix}
  1 & 1 \\
    & 1 & 1 \\
    &   & \ddots & \ddots \\
    &   &        &   1    &  1 \\
    &   &        &        &  -1 \\
\end{bmatrix}.
\]

\begin{figure}[H]
  \labellist\small
  \pinlabel {$x_1$} at 223 66
  \pinlabel {$x_2$} at 204 66
  \pinlabel {$x_{2n-1}$} at 164 66
  \pinlabel {$x_{2n}$} at 147 78
  \endlabellist
  \includegraphics{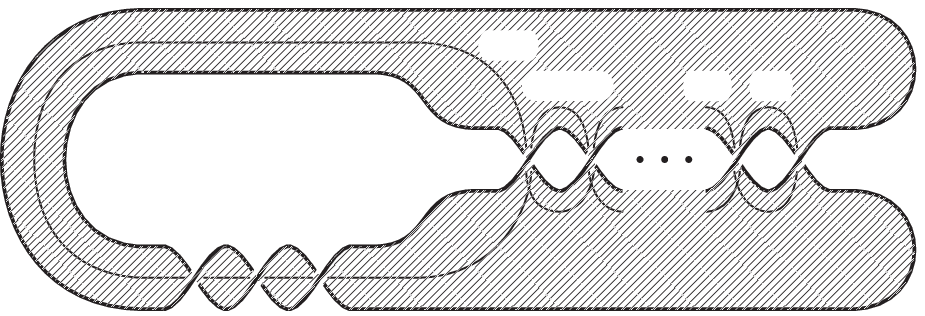}
  \caption{A Seifert surface for the 2-bridge knot $J_n$.}
  \label{figure:two-bridge-seifert-surface}
\end{figure}

Observe that if we replace the bottom-right entry $-1$ by $1$, $A$
becomes the Seifert matrix for the $(2, 2n+1)$ torus knot $T_{2,2n+1}$
shown in Figure~\ref{figure:torus-knot-seifert-surface}.

\begin{figure}[H]
  \labellist\small
  \pinlabel {$x_1$} at 138 65
  \pinlabel {$x_2$} at 119 65
  \pinlabel {$x_{2n-1}$} at 68 65
  \pinlabel {$x_{2n}$} at 43 65
  \endlabellist
  \includegraphics{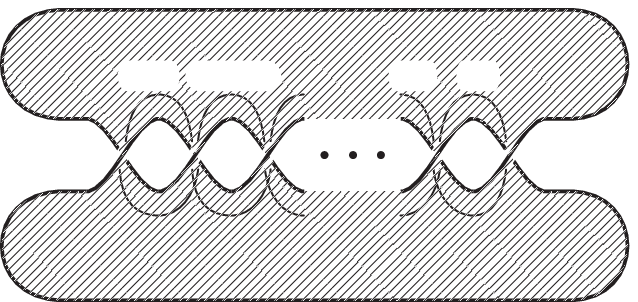}
  \caption{The torus knot $T_{2,2n+1}$ with a Seifert surface.}
  \label{figure:torus-knot-seifert-surface}
\end{figure}

It follows that the Levine-Tristram signatures for $J_n$ and
$T_{2,2n+1}$ differ by at most two, and consequently so do their
averages: for any $d$,
\begin{equation}
  \big|\bar\sigma(J_n,d)-\bar\sigma(T_{2,2n+1},d)\big| \le 2.
  \label{equation:torus-2-bridge-signature-average}
\end{equation}

By work of Litherland~\cite{Litherland:1979-1}, the Levine-Tristram
signature function of a torus knot is well-understood.  Especially,
for the case of $T_{2,2n+1}$, Proposition~1 of
\cite{Litherland:1979-1} gives us
\[
\sigma_{T_{2,2n+1}}(e^{2\pi x \sqrt{-1}}) = 2n-2\Big\lfloor
(2n+1)\Big(\frac12 -x\Big)\Big\rfloor
\]
for rational $x\in (0,\frac12]$.  It follows that
\begin{equation*}
  \begin{aligned}
    \bar\sigma({T_{2,2n+1}},d) &\ge \frac{2}{d}\sum_{k=1}^{(d-1)/2} \bigg(2n - 2\cdot
    (2n+1)\Big(\frac12 - \frac{k}{d} \Big) \bigg) 
    \\
    & = \Big(1-\frac{1}{d^2}\Big)n - \frac{(d-1)^2}{2d^2}
  \end{aligned}
\end{equation*}
for odd~$d$, and
\begin{equation*}
  \begin{aligned}
    \bar\sigma({T_{2,2n+1}},d) &\ge \frac{1}{d}\bigg(2\sum_{k=1}^{(d/2)-1} \bigg(2n - 2\cdot
    (2n+1)\Big(\frac12 - \frac{k}{d} \Big) \bigg) + 2n\bigg)
    \\
    & = n - \frac{d-2}{2d}
  \end{aligned}
\end{equation*}
for even~$d$.  In both cases, we have
\begin{equation}
  \label{equation:torus-signature-average-estimate}
  \bar\sigma({T_{2,2n+1}},d) \ge \Big(1-\frac{1}{d^2}\Big)n - \frac{d-1}{2d}.
\end{equation}
From \eqref{equation:torus-2-bridge-signature-average}
and~\eqref{equation:torus-signature-average-estimate}, we obtain
\begin{equation}
  \bar\sigma(J_n,d) \ge \Big(1-\frac{1}{d^2}\Big)n - \frac{5d-1}{2d}
  \label{equation:signature-averate-estimate-for-J_n}
\end{equation}
where the right hand side is positive for $n\ge 3$, $d\ge 2$.  Now,
combining \eqref{equation:signature-averate-estimate-for-J_n} and
Theorem~\ref{theorem:complexity-lower-bound-from-signature}, we obtain
\[
\begin{aligned}
  c(M_n) &\ge \frac{1}{627419520}\Big(3\Big(1-\frac{1}{d^2}\Big)n -
  \frac{15d-3}{2d}-d+1\Big)
  \\
  &\ge  \frac{1}{627419520}\Big(3\Big(1-\frac{1}{d^2}\Big)n - (d+7)\Big)
\end{aligned}
\]
Together with the last sentence of
Section~\ref{subsection:hyperbolicity-gromov-norm}, it completes the
proof of Theorem~\ref{theorem:gromov-norm-complexity-gap}.

\bibliographystyle{amsalpha}
\renewcommand{\MR}[1]{}
\bibliography{research}

\end{document}